\documentclass[12pt]{amsart}
\usepackage{geometry}                
\usepackage{graphicx}
\usepackage{amssymb}
\usepackage{epstopdf}
\usepackage{color}
\usepackage[colorlinks,  urlcolor=blue]{hyperref}   

\newtheorem{theorem}{Theorem}[section]
\newtheorem{lemma}[theorem]{Lemma}
\newtheorem{corollary}[theorem]{Corollary}

\newtheorem{assumption}{Assumption}

\DeclareSymbolFont{AMSb}{U}{msb}{m}{n}
\DeclareMathSymbol{\N}{\mathbin}{AMSb}{"4E}
\DeclareMathSymbol{\Z}{\mathbin}{AMSb}{"5A}
\DeclareMathSymbol{\R}{\mathbin}{AMSb}{"52}
\DeclareMathSymbol{\Q}{\mathbin}{AMSb}{"51}
\DeclareMathSymbol{\I}{\mathbin}{AMSb}{"49}
\DeclareMathSymbol{\C}{\mathbin}{AMSb}{"43}

\title{Approximate Likelihood Construction for Rough DIfferential Equations}
\author{Anastasia Papavasiliou and Kasia B. Taylor}
\address{Anastasia Papavasiliou \\ Department of Statistics \\ University of Warwick \\ Coventry \\ CV4 7AL \\UK}
\address{Kasia B. Taylor \\ National Grid \\ Warwick Technology Park \\ Gallows Hill \\ Warwick \\ CV34 6DA \\ UK}

\date{\today}                                           
\thanks{Supported by the Leverhulme RPG-2013-270: ``Statistical Inference of Complex Systems through Rough Paths''}

\begin{document}
\maketitle

\begin{abstract}
In this paper, we propose a new framework for the construction of the likelihood of discretely observed differential equations driven by rough paths. The paper is split in two parts: in the first part, we construct the exact likelihood for a discretely observed rough differential equation, driven by a piecewise linear path. In the second part, we use this likelihood to construct approximate likelihoods for discretely observed differential equations driven by a general class of rough paths. Finally, we study the behaviour of the approximate likelihoods when the sampling frequency tends to infinity.
\end{abstract}


\section{Introduction}

\par{Rough differential equations were introduced in the mid-1990s by Terry Lyons (see  \cite{lyons1994differential}, \cite{lyons1995interpretation}
 and \cite{lyons1998differential}).
The theoretical foundations are now
  well established  and have evolved into a self-contained branch of theoretical pursuits,  inspiring the inception of new important developments such as regularity structures (see \cite{hairer2014theory}). 
The theory of rough paths allows one to interpret solutions of  differential equations of the type 
\[ dY_{t}= f(Y_{t})dX_{t}, \]
where $X$  is only of a.s. finite $p$-variation for some $p> 1$, which includes
drivers $X$ that exhibits  rougher paths than Brownian motion. Such equations are called differential equations driven by rough paths (RDEs). The obtained solutions are pathwise. For more information about rough paths theory see e.g.  \cite{lyons2002system}, \cite{friz2010multidimensional}, \cite{friz2014course}. Moreover,  the rough paths setting allows us to encode information about processes in an efficient way (see \cite{levin2013learning} and references therein).}
\par{Fractional  Brownian motion (fBm) $B^{h}$ with Hurst parameter $h$ $ \in (0,1)$ is the most popular non trivial process with a.s. finite $p$-variation for any $p>\frac{1}{h}$.
It was introduced in \cite{kolmogorov1940wienersche}
and popularised by  \cite{mandelbrot1968fractional}. For $h=\tfrac{1}{2}$  it coincides with Brownian motion. Fractional Brownian motion is a  Gaussian self-similar  centred process with  stationary increments and covariance structure 
\[ E(B^{h}_{t}B^{h}_{s})=\tfrac{1}{2} (|t|^{2h}+|s|^{2h}-|t-s|^{2h}).\]
The span of interdependence between  the increments of fBm is infinite, in other words the increments are not strongly mixing (with the exception of $h=\tfrac{1}{2})$. An $n$-dimensional fBM is defined as an ordered set of $n$ independent  scalar fBMs; discussion on multi-dimensional fBMs can be found in e.g. \cite{unterberger2010rough}.}

\par{A classical question for SDEs driven by either Brownian motion and fractional Brownian motion is how to infer unknown parameters in the vector field from discrete observations. 
Recent  surveys  of  methods for inference for SDEs driven by Brownian motion can be found in e.g. \cite{kutoyants2013statistical}, \cite{fuchs2013inference}, \cite{iacus2009simulation}, \cite{bishwal2008parameter}. Reviews of parameter estimation methods for SDEs driven by fractional Brownian can be found in e.g.
\cite{rao2011statistical}, \cite{mishura2008stochastic}. Recent papers on the topic include \cite{beskos2015bayesian}, \cite{saussereau2014nonparametric},  \cite{neuenkirch2014least},  \cite{chronopoulou2013inference},   \cite{brouste2013parameter}, \cite{hu2010parameter},  \cite{tudor2007statistical}. This list is by no means exhaustive.}

\par{ 
As a natural consequence of the introduction of RDEs, a new class of parametric models arises, which includes SDEs driven by either Brownian or fractional Brownian motion. Note that Brownian motion and fractional Brownian motion are considered to be stochastic processes in the context of SDEs while they need to be lifted to rough paths in the RDE context. In this paper, we make a first step towards statistical inference for discretely observed RDEs, by developing a framework for the construction of an approximate likelihood of discretely observed RDEs.  
The discussion of statistical inference methods for RDEs was initiated in  \cite{papavasiliou2011parameter}. The presented  estimation method is based on matching the empirical expected signature  with the theoretical one, where the signature is defined as the set of all iterated integrals of a path. 
  Bayesian inference is conducted in the article \cite{lysy2013statistical} where rough paths approach is adopted and data augmentation technique and Hybrid Monte Carlo are employed.
Stochastic filtering and MLEs in rough paths setting are investigated  in \cite{diehl2014robustness} and related papers \cite{crisan2013robust},  \cite{diehl2013levy}, \cite{diehl2013pathwise}.}



\section{Setting and Main Ideas}

In the first part of the paper, we consider the following type of differential equations
\begin{equation}
\label{eq:main}
dY^{\mathcal D}_t = a(Y^{\mathcal D}_t ; \theta) dt + b(Y^{\mathcal D}_t ; \theta) dX^{\mathcal D}_t,\ \ Y_0 = y_0,\ t\leq T,
\end{equation}
where $X^{\mathcal D}$ is a realisation of a random piecewise linear path in $\R^m$  corresponding to partition ${\mathcal D}$ of $[0,T]$. We also assume that $\theta\in\Theta$, where $\Theta$ is the parameter space. Moreover, we request that $a(\cdot,\theta):\R^d\to \R^d$ and $b(\cdot,\theta):\R^d \to L(\R^m, \R^d)$ are {\rm Lip(1)}, which are sufficient conditions for the existence and uniqueness of the solution $Y^{\mathcal D}$, which is a bounded variation path on $\R^d$.

We will use $I_\theta$ to denote the It\^o map defined by \eqref{eq:main}. That is, $I_\theta$ maps the path $X^{\mathcal D}$ to the path $Y^{\mathcal D}$ and we write
\[ Y^{\mathcal D} = I_\theta (X^{\mathcal D}).\]

First, we develop a framework for performing statistical inference for differential equation \eqref{eq:main}, assuming that we know the distribution of $X^{\mathcal D}$. More precisely, we will aim to construct the likelihood of discrete observations of $Y^{\mathcal D}$ on the grid ${\mathcal D}$, which we will denote by $y_{\mathcal D}$.  The main idea is to use the observations to explicitly construct the It\^o map that maps a finite parametrization of $Y^{\mathcal D}$ to a finite parametrization of $X^{\mathcal D}$. Typically, $Y^{\mathcal D}$ will be parametrized by the observations $y_{{\mathcal D}} := \{ y_{t_i};\ t_i\in {\mathcal D}\}$ and $X^{\mathcal D}$ will be parametrized by the corresponding normalised increments $(\Delta x)_{{\mathcal D}}:=\{ \frac{x_{t_{i+1}} - x_{t_i}}{t_{i+1}-t_i};\ t_i, t_{i+1} \in {{\mathcal D}} \}$. 

In section \ref{section: e-u}, we study the existence and uniqueness of the pair $(X^{\mathcal D}, Y^{\mathcal D})$ for $Y^{\mathcal D}$ parametrised by the given dataset $y_{\mathcal D} = \{y_{t_i};\ t_i \in {\mathcal D}\}$. We give conditions for existence, which are necessary for the methodology to work. Then, we show that for $a$ and $b$ in {\rm Lip}(2) and $b$ non singular, the solution will be unique for the case $m=d$ and it will have $m-d$ degrees of freedom for the case $m>d$. Since existence will not in generally be true for the case $d>m$, this case will not be considered.


In section \ref{section: likelihood}, we explicitly construct the likelihood, treating separately the cases where we have uniqueness and where we have one or more degrees of freedom. 

In the second part of the paper, we consider equation
\begin{equation}
\label{eq:main2}
dY_t = a(Y_t ; \theta) dt + b(Y_t ; \theta) dX_t,\ \ Y_0 = y_0,\ t \leq T,
\end{equation}
where $X\in G\Omega_p(\R^m)$ is the realisation of a random geometric $p$-rough path, defined as the $p$-variation limit of a random sequence of nested piecewise linear paths. Let us denote by ${\mathcal D}(n,T)$ the sequence of nested partitions of $[0,T]$ and by $\pi_n(X)$ the corresponding sequence of piecewise linear paths, such that $d_p(\pi_n(X),X)\to 0$ as $n\to\infty$. We now assume that for each $\theta\in\Theta$, $a(\cdot,\theta)$ and $b(\cdot,\theta)$ are ${\rm Lip(\gamma+1)}$, for some $\gamma>p$, which are sufficient conditions for the existence and uniqueness of the solution $Y=I_\theta(X)\in G\Omega_p(\R^d)$. Moreover, as before, for $b$ non-singular, the pair $(X,Y)$ is unique. If we denote by $Y(n)$ the response to the piecewise linear path $\pi_n(X)$, i.e. $Y(n) = I_\theta\left(\pi_n(X)\right)$, then the continuity of the It\^o map in the $p$-variation topology implies that $d_p(Y(n),Y)\to 0$ as $n\to\infty$.

To simplify notation, we will assume that the partitions ${\mathcal D}(n,T)$ are the dyadic partitions of $[0,T]$, i.e. they are homogeneous with interval size $\delta = 2^{-n}$. We write ${\mathcal D}(n) = \{ k 2^{-n};\ k=0,\dots,N\}$, where $N = 2^n T$.

In section \ref{section: limit}, we use the likelihood constructed before to construct an approximate likelihood of observing a realisation of \eqref{eq:main2} on grid ${\mathcal D}(n)$ for some fixed $n$ -- denoted by $y_{{\mathcal D}(n)}$.  The main idea behind the construction is to replace the model \eqref{eq:main2} that produces the data by \eqref{eq:main}, which is tractable and converges to \eqref{eq:main2} for $n\to\infty$. However, one also needs to normalise the likelihood appropriately, so that the limit still depends on the parameter that we want to estimate.

In section \ref{section: example}, we demonstrate how the method works: we construct the likelihood in the simple case of a discretely observed Ornstein-Uhlenbeck model driven by a picewise linear approximation to fractional Brownian motion and we compute the corresponding limiting likelihoods for $h=\frac{1}{2}$.

Finally, in section \ref{section: convergence}, we make precise in what sense the likelihood constructed in the previous section is approximate. Replacing the complicated model by a simpler one approximating the actual model, when we can construct the likelihood corresponding to the simpler model exactly, is not an uncommon approach for performing statistical inference for otherwise intractable models. For example, this is done in \cite{papavasiliou2009maximum} where the authors replace the actual multiscale model by its limiting diffusion and use that to construct the likelihood. They show that the approximation error due to the mismatch between data (coming from the multiscale model) and model (the limiting equation) disappears in the limit. Following a similar approach, we show that, under suitable conditions, an appropriate distance between the likelihood for discrete observations on a grid ${\mathcal D}(n)$ of the corresponding process $Y(n)$ and of the limiting process $Y$ respectively disappears, as $n\to\infty$. 





\section{Existence and Uniqueness}
\label{section: e-u}

We are given a set of points $y_{\mathcal D}$ in $\R^d$, where ${\mathcal D}$ is the fixed partition of $[0,T]$. In this section, we study the existence and uniqueness of piecewise linear path $X^{\mathcal D}$, whose response $Y^{\mathcal D} = I_\theta(X^{\mathcal D})$ for given $\theta\in\Theta$ goes through points $y_{\mathcal D}$, i.e. $Y^{\mathcal D}_{t_i} = y_{t_i}$ for each $t_i \in {\mathcal D}$. 

First, we discuss how to express $Y^{\mathcal D}$ in terms of $X^{\mathcal D}$. By construction, $X^{\mathcal D}$ is linear between grid points, i.e.
\[ X^{\mathcal D}_t = X^{\mathcal D}_{t_i} + \Delta X^{\mathcal D}_{t_i} \left( t-t_i \right),\ \forall t \in [t_i,t_{i+1}),\ t_i,t_{i+1}\in{\mathcal D},\]
where $\Delta X^{\mathcal D}_{t_i} = \frac{X^{\mathcal D}_{t_{i+1}}-X^{\mathcal D}_{t_i}}{t_{i+1}-t_i}$.
By definition, $Y^{\mathcal D} = I_\theta(X^{\mathcal D})$ which implies that for every $ t \in [t_i,t_{i+1})$, $Y^{\mathcal D}_t$ satisfies
\begin{eqnarray*}
dY^{\mathcal D}_t &=& a(Y^{\mathcal D}_t ; \theta) dt + b(Y^{\mathcal D}_t ; \theta) dX^{\mathcal D}_t = \\
&=& \left( a(Y^{\mathcal D}_t ; \theta) dt + b(Y^{\mathcal D}_t ; \theta)  \Delta X^{\mathcal D}_{t_i} \right) dt
\end{eqnarray*}
with initial conditions $Y^{\mathcal D}_{t_i} = y_{t_i}$. This is an ODE and we have already assumed sufficient regularity on $a$ and $b$ for existence and uniqueness of its solutions. The general form of the ODE is given by
\begin{equation}
\label{general ODE} 
d\tilde{Y}_t = \left(  a(\tilde{Y}_t;\theta) + b(\tilde{Y}_t;\theta)\cdot c \right) dt, Y_{0} = y_0
\end{equation}
and we will denote its solution by $F_t(y_0,c;\theta)$. Then, 
\begin{equation}
\label{y(n) general formula} 
Y^{\mathcal D}_t = F_{t-t_i}(y_{t_i}, \Delta X_{t_i};\theta),\ \forall t\in[t_i,t_{i+1}).
\end{equation}
In order to fit $Y^{\mathcal D}$ to the observed data $y_{{\mathcal D}}$, we need to solve for $\Delta X_{t_i}$, using the terminal value, i.e. solve
\begin{equation}
\label{solving for c}
F_{t_{i+1}-t_i}(y_{t_i},\Delta X_{t_i};\theta) = y_{t_{i+1}}
\end{equation}
for $\Delta X_{t_i}(y_{t_i},y_{t_{i+1}};\theta)$. So, for every interval $[t_i,t_{i+1})$, we need to solve an independent system of $d$ equations and $m$ unknowns. That is, we need to study the existence and uniqueness of solutions with respect to $c$ of the system
\begin{equation}
\label{system} 
F_\delta(y_0, c; \theta ) = y_1 ,
\end{equation}
for every $\theta$ and for appropriate values of $\delta, y_0$ and $y_1$. We are going to assume existence of solution, by requiring that $y_1 \in \cap_{\theta\in\Theta}{\mathcal M}_\delta\left( y_0;\theta \right)$, where
\begin{equation}
\label{solution manifold} 
{\mathcal M}_\delta\left( y_0;\theta \right) = \left \{ F_\delta(y_0,c;\theta); c\in \R^m \right\}.
\end{equation}
Now, suppose that $c_1$ and $c_2$ are both solutions for a given $\theta\in\Theta$, i.e.
\[ F_\delta(y_0,c_1;\theta) = y_1 = F_\delta(y_0,c_2;\theta).\]
We can write the difference as
\begin{eqnarray*}
F_\delta(y_0,c_2;\theta) - F_\delta(y_0,c_1;\theta) = \left( \int_0^1 D_c F_\delta(y_0,c_1 + s(c_2-c_1);\theta) ds \right) \cdot (c_2-c_1).
\end{eqnarray*}
Thus, $F_\delta(y_0, c_1; \theta ) = F_\delta(y_0, c_2; \theta )$ implies
\[  \left( \int_0^1 D_c F_\delta(y_0,c_1 + s(c_2-c_1);\theta) ds \right) \cdot (c_2-c_1) = 0.\]
So, it is sufficient to show that $\forall \xi\in\R^m$, the rank of $d\times m$ matrix $D_c F_\delta(y_0,\xi;\theta)$ is $d$, which implies that the solution will have $m-d$ degrees of freedom, i.e. given $m-d$ coordinates of $c$, the other coordinates are uniquely defined. In particular, for $d=m$ we get uniqueness.

Since the vector field of \eqref{general ODE} is linear with respect to $c$, we know that $F_t(y_0,c;\theta)$ will be continuously differentiable with respect to $c$ for every $y_0$, $\theta$ and $t$ in the appropriate bounded interval \cite{kelley2001difference}. 
Thus, we define a new auxiliary process as $Z_t(c) = D_c F_t(y_0,c;\theta) \in \R^{d\times m}$, or, 
\begin{equation}
\label{Z definition}
Z^{i,\alpha}_t(c) = \frac{\partial}{\partial c_\alpha} F^i_t(y_0,c; \theta),\ {\rm for}\ i = 1,\dots,d,\ \alpha = 1,\dots,m.
\end{equation}
Then, assuming one additional degree of regularity, $Z_t(c)$ satisfies
\begin{eqnarray*}
\frac{d}{dt} Z^{i,\alpha}_t(c) &=& \frac{d}{dt} \frac{\partial}{\partial c_\alpha} F^i_t(y_0,c; \theta) = \frac{\partial}{\partial c_\alpha} \frac{d}{dt} F^i_t(y_0,c; \theta) = \\
&=& \frac{\partial}{\partial c_\alpha}\left( a_i(F_t(y_0,c; \theta) ) + \sum_{\beta=1}^m c_\beta b_{i\beta}(F_t(y_0,c; \theta) )\right) = \\
&=& \sum_{j=1}^d\left( \partial_j a_i(F_t(y_0,c; \theta) ) + \sum_{\beta=1}^m c_\beta  \partial_j b_{i\beta}(F_t(y_0,c; \theta) ) \right)  {\bar Z}^{j\alpha}_t(c) + {b}_{i\alpha}(F_t(y_0,c; \theta) ),
\end{eqnarray*}
where by $\bar{Z}_t^\alpha(c)$ we denote column $\alpha\in\{1,\dots,m\}$ of matrix $Z_t(c)$. More concisely, we write
\begin{equation}
\label{Z equation} 
\frac{d}{dt} \bar{Z}_t^\alpha(c)  = \bigtriangledown\left(  a+ b\cdot c \right)(F_t;\theta)\cdot \bar{Z}_t^\alpha(c) + \bar{b}_\alpha(F_t;\theta),
\end{equation}
where $\bigtriangledown f $ of a function $f:\R^d\to\R^d$ we denote the $d\times d$ matrix defined as
\begin{equation*}
\left( \bigtriangledown f (y)\right)_{i,j} = \partial_j f_i(y).
\end{equation*}
Also, $\bar{b}_\alpha$ is column $\alpha$ of matrix $b$. Note that, for each fixed $\alpha$, this is a linear equation of $\bar{Z}^\alpha(c)$ with non-homogeneous coefficients. Also note that the initial conditions will be
\[ Z^{i,\alpha}_0(c) = \frac{\partial}{\partial c_\alpha} F^i_0(y_0,c; \theta) = \frac{\partial}{\partial c_\alpha} y_0 \equiv 0,\ \forall i= 1,\dots,d,\ \alpha= 1,\dots,m.\]
Thus, the solution to this equation will be
\begin{equation}
\label{Z solution}
\bar{Z}^\alpha_t(c) = \int_0^t \exp\left( {\bf A}\right)_{s,t} \bar{b}_\alpha(F_s;\theta) ds,
\end{equation}
where by $\exp\left( A\right)_{s,t}$ we denote the sum of iterated integrals
\begin{equation*}
\exp\left( {\bf A}\right)_{s,t} = \sum_{k=0}^\infty {\bf A}^k_{s,t}
\end{equation*}
and
\[ {\bf A}^k_{s,t} = \int\cdots\int_{s<u_1<\dots<u_k<t} A(F_{u_1};\theta)\cdots A(F_{u_k};\theta) du_1 \dots du_k
\]
for 
\begin{equation}
\label{def: A}
A(y;\theta) = \bigtriangledown\left(  a+  b\cdot c \right)(y;\theta)
\end{equation}
This is a $d\times d$ matrix and for $k=0$ we get the identity matrix, i.e. ${\bf A}^0_{s,t} = I_d$. Since each vector $\bar{Z}^\alpha_\delta$ is a column of the matrix $D_c F_\delta(y_0,c;\theta)$, the condition that the rank of this matrix is $d$ is equivalent to $d$ columns being linearly independent. Without loss of generality, let's consider the first $d$ columns ($d\leq m$) and let us assume that 
\begin{equation}
\label{linear combination} 
\lambda_1 \bar{Z}^1_\delta + \cdots + \lambda_d \bar{Z}^d_\delta = \bar{0},
\end{equation}
for some $\lambda_1,\dots,\lambda_d \in \R$. We need to find conditions such that \eqref{linear combination} is equivalent to $\lambda_1 = \cdots = \lambda_d = 0$. Using \eqref{Z solution} we get that \eqref{linear combination} is equivalent to
\begin{eqnarray*}
\int_0^\delta \exp\left( {\bf A}\right)_{s,\delta} \left( \lambda_1 \bar{b}_1(F_s;\theta) + \cdots + \lambda_d \bar{b}_d(F_s;\theta)\right) ds = \bar{0}.
\end{eqnarray*}
Using the continuity of the integrated function with respect to $s$, 
we can deduce that there exists a $\delta^\prime\in[0,\delta]$, such that we can write the above relationship as
\begin{eqnarray*}
\exp\left( {\bf A}\right)_{\delta^\prime,\delta} \left( \lambda_1 \bar{b}_1(F_{\delta^\prime};\theta) + \cdots + \lambda_d \bar{b}_d(F_{\delta^\prime};\theta)\right) \cdot \delta = \bar{0}.
\end{eqnarray*}
It is known that $\exp\left( {\bf A}\right)_{\delta^\prime,\delta}$ is invertible, with inverse equal to $\exp\left( {\bf A}\right)_{\delta,\delta^\prime}$. Consequently, the above relationship can only be true if
\[ \lambda_1 \bar{b}_1(F_{\delta^\prime};\theta) + \cdots + \lambda_d \bar{b}_d(F_{\delta^\prime};\theta) = \bar{0}. \]
Assuming that the rank of $d\times m$ matrix $b(y;\theta)$ is $d$ for every $y$, this implies that $\lambda_1 = \cdots = \lambda_d = 0$, which is what we required. 

We have shown the following results:
\begin{lemma}
\label{rank Z}
Suppose that ${\rm rank}\left(b(y,\theta)\right) = d$ for every $y$ and that $a(\cdot,\theta)$ and $b(\cdot,\theta)$ are ${\rm Lip}(2)$. Then
\[ {\rm rank}\left(Z_t(c)\right) = {\rm rank}\left( D_c F_t(y_0,c;\theta) \right) = d. \]
\end{lemma}

Note that the construction of the process $Z$ can also be done for $X\in G\Omega_p(\R^m)$, provided that its piecewise linear approximations converge in $p$-variation and that the vector field functions $a$ and $b$ are now ${\rm Lip}(\gamma +1)$. Uniqueness of the pair $(X,Y)$ for given $Y$ follows by taking limits. We make this statement formal in the following 
\begin{corollary}
Suppose that ${\rm rank}\left(b(y;\theta)\right) = d$ for every $y$ and that $a(\cdot;\theta)$ and $b(\cdot;\theta)$ are ${\rm Lip}(\gamma + 1)$. Then, for a given $Y$, the solution $(X,Y)$ of \eqref{eq:main2} is unique.
\end{corollary}


\section{Construction of the Likelihood}
\label{section: likelihood}
In this section, we construct the exact likelihood of observing the process $Y^{\mathcal D}$ on a fixed grid ${\mathcal D}$, denoted by $y_{\mathcal D} = Y^{\mathcal D}_{\mathcal D}$, where $Y^{\mathcal D}$ is the response to a piecewise linear path $X^{\mathcal D}$ on ${\mathcal D}$ through \eqref{eq:main}. The key realisation is that the values of $Y^{\mathcal D}$ on ${\mathcal D}$ actually completely describe the process $Y^{\mathcal D}$.

First, we need to impose a probability structure to the space. Let $\left(\Omega,{\mathcal F}, {\mathbb P}\right)$ be a probability space and let $X^{\mathcal D}$ be a random variable, taking values in the space of piecewise linear paths on ${\mathcal D}$, equipped with the 1-variation topology. So, $X^{\mathcal D}$ is a random piecewise linear path on $\R^m$ corresponding to partition ${\mathcal D}$. Thus, it is fully described by the distribution of its values on the grid ${\mathcal D}$, or, equivalently, its increments. Let us denote that distribution by ${\mathbb P}_{\Delta X_{\mathcal D}}$.

The measure ${\mathbb P}_{\Delta X_{\mathcal D}}$ is a distribution on the finite dimensional space $\R^{m\times N}$, with $N=|{\mathcal D}|$ being the size of the partition. We will assume that this is absolutely continuous with respect to Lebesgue. 

By the continuity of $I_\theta$, $Y^{\mathcal D} = I_\theta (X^{\mathcal D})$ is also an implicitly finite dimensional random variable, whose distribution can be fully describe by the probability of its values on the grid. Below, we construct the likelihood of observing a realisation of $Y^{\mathcal D}$, corresponding to parametrisation $y_{\mathcal D}$. 

\subsection{Case I: Uniqueness}

 Let us first consider the case where we have existence and uniqueness of solutions to system \eqref{system}, so $m=d$. Then, for each dataset ${y}_{{\mathcal D}}$, the set $\{\Delta X_{t_i}(y_{t_i},y_{t_{i+1}};\theta), t_i\in{\mathcal D}\}$ will be uniquely defined as the collection of solutions of \eqref{system}. This defines a map
\begin{equation}
\label{inverse Ito projection}  
I_{\theta,{\mathcal D}}^{-1} (y_{{\mathcal D}}) = \{ \Delta X_{t_i}(y_{t_i},y_{t_{i+1}};\theta),\ t_i\in{\mathcal D} \},
\end{equation}
which can be viewed as a transformation of the observed random variable in terms of the increments of the driving noise. Note that ${y}_{{\mathcal D}}$ and $\{\Delta X_{t_i}, t_i\in{\mathcal D}\}$ fully parametrize processes $Y^{\mathcal D}$ and $X^{\mathcal D}$. Thus, we can write the likelihood of observing $y_{{\mathcal D}}$ as 
\begin{equation*}
\label{likelihood 1}
L_{Y^{\mathcal D}}\left(y_{{\mathcal D}}|\theta\right) = L_{\Delta X_{\mathcal D}}\left(  I_{\theta,{\mathcal D}}^{-1} (y_{{\mathcal D}})\right) |D I_{\theta,{\mathcal D}}^{-1} (y_{{\mathcal D}})|,
\end{equation*}
where by $L_{\Delta X_{\mathcal D}}(\Delta x_{{\mathcal D}})$ we denote the Radon-Nikodym derivative of ${\mathbb P}_{\Delta X_{\mathcal D}}$ with respect to Lebesque. This will be explicitly known since we assumed that we know the distribution of $X^{\mathcal D}$. Finally, since $\Delta X_{t_i}$ only depends on $y_{t_i}$ and $y_{t_{i+1}}$ and not the whole path, it is not hard to see that the Jacobian matrix will be block lower triangular and consequently, the determinant will be the product of the determinants of the blocks on the diagonal:
\begin{equation}
\label{Jacobian}
 |D I_{\theta,{\mathcal D}}^{-1} (y_{{\mathcal D}})| = \prod_{t_i \in{\mathcal D}} \left| \triangledown \Delta X_{t_i}(y_{t_i},y;\theta)|_{y = y_{t_{i+1}}} \right|.
\end{equation}
Note that, by definition,
\[ F_{t_{i+1}-t_i}(y_{t_i},\Delta X_{t_i}(y_{t_i},y;\theta);\theta)\equiv y.
\]
Thus,
\[ D_c F_{t_{i+1}-t_i}(y_{t_i},c;\theta)|_{c = \Delta X_{t_i}(y_{t_i},y_{t_{i+1}};\theta)}\cdot \triangledown\Delta X_{t_i}(y_{t_i},y;\theta)|_{y = y_{t_{i+1}}}\equiv I_d
\]
and, consequently,
\begin{eqnarray*}
\triangledown\Delta X_{t_i}(y_{t_i},y;\theta)|_{y = y_{t_{i+1}}} &=& \left( D_c F_{t_{i+1}-t_i}(y_{t_i},c;\theta)|_{c = \Delta X_{t_i}(y_{t_i},y_{t_{i+1}};\theta)} \right)^{-1} = \\
&=& \left( Z_{t_{i+1}-t_i}(\Delta X_{t_i}(y_{t_i},y_{t_{i+1}};\theta)) \right)^{-1}.
\end{eqnarray*}
So, the likelihood can be written as
\begin{equation}
\label{likelihood 1}
L_{Y^{\mathcal D}}\left(y_{{\mathcal D}}|\theta\right) = L_{\Delta X_{\mathcal D}}\left(  I_{\theta,{\mathcal D}}^{-1} (y_{{\mathcal D}})\right)\left( \prod_{t_i\in{\mathcal D}} \left| Z_{t_{i+1}-t_i}\left(  I_{\theta,{\mathcal D}}^{-1} (y_{{\mathcal D}})_{t_i}\right)\right|\right)^{-1}.
\end{equation}

\subsection{Case II: Degrees of Freedom}

Now suppose that $m>d$. Without loss of generality, let us assume that given coordinates $c_{d+1},\dots,c_m$, the remaining coordinates $c_1,\dots,c_d$ are uniquely defined. Similar to previous case, we denote by $I^{-1}_{\theta,{\mathcal D},c_{d+1},\dots,c_m} (y_{{\mathcal D}})$ the map from data points ${y}_{{\mathcal D}}$ to the first $d$ increments, denoted by $\{\Delta X_{t_i}(y_{t_i},y_{t_{i+1}};\theta,c_{d+1},\dots,c_m)^i, t_i\in{\mathcal D},i=1,\dots,d\}$, for fixed $c_{d+1},\dots,c_m$. As before, this can be viewed as a transformation of the observed random variable in terms of the first $d$ increments of the driving noise and we get a similar formula for the likelihood:
\begin{eqnarray*}
L_{Y^{\mathcal D}}\left(y_{{\mathcal D}}|\theta,c_{d+1},\dots,c_m\right) = &\\
L_{\Delta X_{\mathcal D}}\left( I^{-1}_{\theta,{\mathcal D},c_{d+1},\dots,c_m} (y_{{\mathcal D}})\right)\cdot &
\left( \prod_{t_i\in{\mathcal D}} \left| Z_{t_{i+1}-t_i}\left( I^{-1}_{\theta,{\mathcal D},c_{d+1},\dots,c_m} (y_{{\mathcal D}})_{t_i}\right)\right|\right)^{-1}.
\end{eqnarray*}
However, $c_{d+1},\dots,c_m$ will not be known in general, so we have to consider all possible values of them, leading to the formula
\begin{eqnarray*}
L_{Y^{\mathcal D}}\left(y_{{\mathcal D}}|\theta\right) = \int_{\R^{m-d}} L_{Y^{\mathcal D}}\left(y_{{\mathcal D}}|\theta,x_{d+1},\dots,x_m\right) {\mathbb P}_{c_{d+1},\dots,c_{m}}\left( dx_{d+1},\dots,dx_{m}\right) =\\
 \int_{\R^{m-d}} L_{\Delta X_{\mathcal D}}\left( I^{-1}_{\theta,{\mathcal D},x_{d+1},\dots,x_m} (y_{{\mathcal D}})\right)\cdot 
\left( \prod_{t_i\in{\mathcal D}} \left| Z_{t_{i+1}-t_i}\left( I^{-1}_{\theta,{\mathcal D},x_{d+1},\dots,x_m} (y_{{\mathcal D}})_{t_i}\right)\right|\right)^{-1} \cdot \\
\cdot {\mathbb P}_{c_{d+1},\dots,c_{m}}\left( dx_{d+1},\dots,dx_{m}\right),
\end{eqnarray*}
where ${\mathbb P}_{c_{d+1},\dots,c_{m}}$ is the marginal distribution of ${\mathbb P}_{\Delta X_{\mathcal D}}$ on $\R^{(m-d)\times N}$. 


\section{The Limiting Case}
\label{section: limit}

In the first part of the paper, we assumed that we observe the response to a differential equation driven by a piecewise linear path \eqref{eq:main} and we constructed the exact likelihood of the observations. In this second part of the paper, we lift the assumption that the driver $X$ is a piecewise linear path; instead, we assume that we discretely observe the response to a differential equation \eqref{eq:main2} driven by a $p$-rough path $X$. We aim to construct an approximate likelihood for the observations and we study the behaviour of the approximate likelihood when the sampling size $\delta\to 0$. A crucial assumption is that there exists a sequence of partitions ${\mathcal D}(n)$ (usually dyadic) such that the corresponding piecewise linear interpolations $\pi_n(X)$ of the path $X$ converge in $p$-variation to the $p$-rough path $X$. This allows us to replace \eqref{eq:main2} by \eqref{eq:main}. 

Let us denote by $y_{\mathcal D(n)}$ the sequence of observations of the limiting equation \eqref{eq:main2} on the grid ${\mathcal D}(n)$. We will use the likelihood $L_{Y^{\mathcal D(n)}}$ constructed in \eqref{likelihood 1} to construct an approximate likelihood for the partially observed limiting equation -- for simplicity, we will now denote it by $L_{Y(n)}$. Also, to simplify the exposition, we will focus on the case where we have uniqueness, i.e. $m=d$ and $b$ is non-singular.

A first idea would be to define the approximate likelihood as $L_{Y(n)}\left( y_{\mathcal D(n)} | \theta \right)$. Then, we would hope to show that, for $n$ large, this will be close to $L_{Y(n)}\left( y(n)_{\mathcal D(n)} | \theta \right)$ in a way that allows the estimators constructed using this likelihood to inherit a lot of the properties of those constructed using exact likelihood $L_{Y(n)}\left( y(n)_{\mathcal D(n)} | \theta \right)$. Note that the difference between $y_{\mathcal D(n)}$ and $y(n)_{\mathcal D(n)}$ is that the first is the response to a realisation of the rough path $x$ while the latter is the response to the piecewise linear approximation of $x$ on the grid ${\mathcal D}(n)$, i.e. $y(n)_{{\mathcal D}(n)} = I_\theta(\pi_n(x))_{\mathcal D(n)}$, making the likelihood exact. Note that the two sequences converge in $p$-variation, for $n\to \infty$. So, we expect that the estimators constructed using the datasets $y_{\mathcal D(n)}$ and $y(n)_{\mathcal D(n)}$ will be close, provided that the estimator is continuous in the $p$-variation topology. 

However, when the model involves more than one parameter, it is often the case that, in the limit, $L_{Y(n)}\left( y_{\mathcal D(n)} | \theta \right)$ as a function of $\theta$ scales differently for different coordinates of $\theta$. In particular, this occurs because the drift component $dt$ scales differently than the `diffusion' component $dX_t$. Thus, if we use the likelihood to construct the Maximum Likelihood Estimators (MLEs), we need to carefully normalise the likelihood appropriately, depending on which coordinate of $\theta$ we want to estimate at any time. Actually, it is equivalent and more convenient to work with the log-likelihood: normalising the log-likelihood involves adding functions to the log-likelihood that are independent of the parameters we want to estimate. So, we want to construct an expansion of the log-likelihood of the form
\begin{equation}
\label{likelihood expansion}
\ell_{Y(n)} \left(y_{{\mathcal D}(n)}|\theta\right) = \sum_{k=0}^M \ell_{Y(n)}^{(k)} \left(y_{{\mathcal D}(n)}|\theta\right) N^{-\alpha_k} + R_M(y_{{\mathcal D}(n)},\theta)
\end{equation}
for $N = T/\delta=2^n T$, some $M\in\N$ and $-\infty<\alpha_0 <\alpha_1 <\dots <\alpha_M<\infty$, where $\ell_{Y(n)}^{(k)} \left(y_{{\mathcal D}(n)}|\theta\right)$ converge to a non-trivial limit (finite and non-zero) for every $k=0,\dots,M$ and the remainder $R_M(y_{{\mathcal D}(n)},\theta)$ satisfies 
\begin{equation}
\label{eq:remainder}
\lim_{N\to\infty}\sup_{\theta\in\Theta} N^{\alpha_M} R_M(y_{{\mathcal D}(n)},\theta)=0, {\rm a.s.}
\end{equation}
as $N\to\infty$. This will exist, assuming sufficient smoothness of the log-likelihood function of $\Delta x_{{\mathcal D}(n)}$, but will not necessarily be unique, as the lower orders can `hide' high order terms. To ensure that this is not the case, we make the following additional assumption.

\begin{assumption}
\label{convergence of scales}
Suppose that the log-likelihood function satisfies \eqref{likelihood expansion}. We will also assume that for almost every pair $(y,\tilde{y})$,
\begin{equation}
\label{convergence rate of pl approx}  
\lim_{n\to\infty} N(n)^{\alpha_M} | \ell_{Y(n)} \left(y_{{\mathcal D}(n)}|\theta\right) - \ell_{Y(n)} \left(\tilde y_{{\mathcal D}(n)}|\theta\right) |  = 0
\end{equation}
implies that
\begin{equation}
\label{scaled logL modulus of continuity}
\lim_{n\to\infty} | \ell^{(k)}_{Y(n)} \left(y_{{\mathcal D}(n)}|\theta\right) - \ell_{Y(n)}^{(k)} \left(\tilde y_{{\mathcal D}(n)}|\theta\right) | = 0,\ \forall k=0,\dots,M.
\end{equation}
\end{assumption}
Consider the case where $M=1$. Then \eqref{convergence rate of pl approx} and \eqref{eq:remainder} lead to 
\[ \left(\ell^{(0)}_{Y(n)} \left(y_{{\mathcal D}(n)}|\theta\right) - \ell^{(0)}_{Y(n)} \left(\tilde y_{{\mathcal D}(n)}|\theta\right)\right)N^{-\alpha_0} + \left(\ell^{(0)}_{Y(n)} \left(y_{{\mathcal D}(n)}|\theta\right) - \ell^{(0)}_{Y(n)} \left(\tilde y_{{\mathcal D}(n)}|\theta\right)\right)N^{-\alpha_1}\to 0
\]
as $n\to\infty$. From this, we can deduce that the $0^{\rm th}$ order goes to zero, i.e.
\[ \left(\ell^{(0)}_{Y(n)} \left(y_{{\mathcal D}(n)}|\theta\right) - \ell^{(0)}_{Y(n)} \left(\tilde y_{{\mathcal D}(n)}|\theta\right)\right)\to 0
\]
but not necessarily the $1^{\rm st}$ order. Instead, we get that
\[ \lim_{n\to\infty}\left(\ell^{(0)}_{Y(n)} \left(y_{{\mathcal D}(n)}|\theta\right) - \ell^{(0)}_{Y(n)} \left(\tilde y_{{\mathcal D}(n)}|\theta\right)\right)N^{\alpha_1-\alpha_0} = \lim_{n\to\infty} \left(\ell^{(0)}_{Y(n)} \left(y_{{\mathcal D}(n)}|\theta\right) - \ell^{(0)}_{Y(n)} \left(\tilde y_{{\mathcal D}(n)}|\theta\right)\right).\]

These limits, if non zero, will usually be random and will depend on $(y,\tilde{y})$. The probability of them being equal will be $0$, unless they actually converge to the same random variable. This will happen if the $0^{\rm th}$ order `hides' a $1^{\rm st}$ order component. Assumption \ref{convergence of scales} makes sure that this cannot happen. 

Now, let $\theta_i$ be an arbitrary coordinate of the parameter $\theta$ and suppose that $\ell_{Y(n)}^{(m)} \left(y_{{\mathcal D}(n)}|\theta\right)$ is the first component of the likelihood containing sufficient information for estimating $\theta_i$. Note that information content will also depend on the estimation method, as discussed below. 
Intuitively, we expect that the first $m-1$ components will be irrelevant to the estimation of the parameter and should be ignored, while the remaining log-likelihood should be normalised by $N^{\alpha_m}$. Thus, we will say that coordinate $\theta_i$ of the parameter is of order $m$ and, given observations of the limiting equation, we will use the dominating term $\ell_{Y(n)}^{(m)} \left(y_{{\mathcal D}(n)}|\theta\right)$ for its estimation. We will assume that all coordinates of the parameter are of finite order -- otherwise, they cannot be estimated! 


Below, we use this framework to build estimators for parameter $\theta$ using the constructed likelihoods. We discuss separately the two most common approaches, corresponding to the Frequentist or Bayesian paradigm. 

\subsection{Frequentist Setting}

In the frequentist setting, we use the likelihood constructed above in order to construct the MLE of the parameter $\theta\in\Theta$. Let us assume that the likelihood is a differentiable function of the parameter $\theta$. Below, we describe how to inductively define the MLEs of different co-ordinates of $\theta$.
\begin{itemize}
\item[1.]  We say that co-ordinates of the parameter $\theta$ are of order $\alpha_0$ and we denote them by $\theta(0)$ if 
\[ \nabla_{\theta(0)}\ell^{(0)}(y_{\mathcal D(n)} | \theta) \not\equiv \bar{0}.\]
Then, we define their estimate as 
\[ \hat{\theta}(0,y_{{\mathcal D}(n)}) = {\rm argmax}_{\theta(0)} \ell^{(0)}_{Y(n)}\left(y_{{\mathcal D}(n)}|\theta(0) \right).\]
\item[2.] Suppose that we have defined parameters of order up to $m$ and their MLEs, for some $m\geq 0$. Then, we define $\theta(m+1)$ as the set of all coordinates of $\theta$ that are not included in $\theta(0),\dots,\theta(m)$, that satisfy
\[ \nabla_{\theta(m+1)}\ell^{(m+1)}(y_{\mathcal D(n)} | \hat\theta(0,y_{{\mathcal D}(n)}),\dots,\hat\theta(m,y_{{\mathcal D}(n)}),c\theta(m))\not\equiv\bar{0},\]
where $c\theta(m)$ is the set of all coordinates that are not of order $\leq m$.
We define their MLE as
\[ \hat{\theta}_{m+1}(y_{{\mathcal D}(n)}) = {\rm argmax}_{\theta_{m+1}} \ell^{(m+1)}_{Y(n)}\left(y_{{\mathcal D}(n)}|\hat\theta_0(y_{{\mathcal D}(n)}),\dots,\hat\theta_m(y_{{\mathcal D}(n)}),\theta_{m+1} \right).\]
\end{itemize}

\subsection{Bayesian Setting}

In the Bayesian setting, we use the likelihood constructed above together with a prior distribution on the parameter space that we will denote by $u$, in order to construct the posterior distribution of the parameter $\theta\in\Theta$. In this setting, it is not necessary to separate the likelihood into different scalings. However, doing so can shed more light into the process, so below we describe the way that this can be done:
\begin{itemize}
\item[1.] First, we start with the lower order $\alpha_0$ and work our way up. We say that a co-ordinate $\theta_i$ of the parameter is of order $\alpha_k$ if the distance between the posterior and prior on $\theta_i$, for an appropriate choice of distance on the measure space, is non-zero for the first time when the posterior is computed using the scaling of the likelihood corresponding to $\ell^{(k)}_{Y(n)}\left(y_{{\mathcal D}(n)}|\theta_0 \right)$ for any other value of the parameter except for a set of measure zero with respect to the prior. We will denote by $\theta(k)$ all the co-ordinates of the parameter that are of order $k$ and by $r$ the maximum order. 
\item[2.] Starting with the lower order, we construct the posterior inductively. We define the posterior at level $k$, denoted by $u_k$ as the posterior computed using $u_{k-1}$ as the prior and
\[ \exp\left(\ell^{(k)}_{Y(n)}\left(y_{{\mathcal D}(n)}|\theta\right)N^{-\alpha_k}\right)\]
as the likelihood, for all parameters of order up to $k$ and $u_0 = u$.
\end{itemize}


\section{Example: The 1d fractional O.U. process}
\label{section: example}

To demonstrate the methodology, we will apply the ideas described in the previous section to a simple example. We consider the differential equation
\begin{equation}
\label{1d fOU}
dY^{\mathcal D}_t = -\lambda Y^{\mathcal D}_t dt + \sigma X^{\mathcal D}_t,\ \ Y^{\mathcal D}_0 = 0,
\end{equation}
where $X^{\mathcal D}_t$ is the piecewise linear interpolation to a fractional Brownian path with Hurst parameter $h$ on a homogeneous grid ${\mathcal D} = \{ k\delta;
 k=0,\dots,N\}$ where $N\delta = T$. Our goal will be to construct the likelihood of discretely observing a realisation of the solution $Y^{\mathcal D}(\omega)$ on the grid, for parameter values $\theta = (\lambda,\sigma)\in\R_+\times\R_+$. 
 
 Our first task is to explicitly construct the parametrization of $Y^{\mathcal D}(\omega)$ in terms of its values on the grid $y_{\mathcal D}$, that completely determine the process. Let $X^{\mathcal D}(\omega)$ be the piecewise linear interpolation on ${\mathcal D}$ of the corresponding realisation of a fractional Brownian path driving \eqref{1d fOU}. We will denote by $x_{t_i}$ its values on the grid, i.e. $X^{\mathcal D}(\omega)_{t_i} = x_{t_i},\ \forall t_i \in {\mathcal D}$. Since $X^{\mathcal D}(\omega)$ is the piecewise linear path defined on these points, $Y^{\mathcal D}(\omega)$ will be the solution to
\[ dY^{\mathcal D}(\omega)_t = -\lambda Y^{\mathcal D}(\omega)_t dt + \sigma \frac{x_{(k+1)\delta} - x_{k \delta}}{\delta} dt, \]
which is given by
\[ Y^{\mathcal D}(\omega)_t = Y^{\mathcal D}(\omega)_{k\delta} e^{-\lambda (t-k\delta)} + \frac{\sigma}{\lambda} \frac{x_{(k+1)\delta} - x_{k \delta}}{\delta}\left( 1- e^{-\lambda (t-k\delta)} \right), t\in [k\delta, (k+1)\delta).\]
We now need to solve for the unknown $\Delta x_{k+1} := x_{(k+1)\delta} - x_{k \delta}$: for $t = (k+1)\delta$. We get
\begin{equation}
\label{fOU:discreteI}
y_{(k+1)\delta} = y_{k\delta} e^{-\lambda \delta} + \frac{\sigma \Delta x_{k+1}}{\lambda\delta} \left( 1- e^{-\lambda \delta} \right)
\end{equation}
and, consequently,
\begin{equation}
\label{fOU:discreteInvI}
I^{-1}_{\theta,{\mathcal D}}(y_{{\mathcal D}})_{k+1}:=\Delta x_{k+1} = \frac{\lambda\delta\left( y_{(k+1)\delta} - y_{k\delta} e^{-\lambda \delta} \right)}{\sigma \left( 1-e^{-\lambda \delta} \right)},\ k=0,\dots,N-1,
\end{equation}
with $y_0 = 0$ and $\theta = (\lambda,\sigma)$. Thus, $Y^{\mathcal D}(\omega)$ is given by
\begin{equation}
\label{fOU:Y(n)}
Y^{\mathcal D}(\omega)_t = y_{k \delta} e^{-\lambda (t-k\delta)} + \frac{y_{(k+1)\delta} - y_{k\delta} e^{-\lambda \delta}}{1-e^{-\lambda \delta}} \left( 1- e^{-\lambda (t-k\delta)} \right),
\end{equation}
for $t \in [k\delta, (k+1)\delta)$ and $y_0 = 0$.  

Clearly, in this case, the solution of system \eqref{system} always exists and is unique under the condition that $\sigma\neq 0$. Let us now compute the process $Z$ defined in \eqref{Z definition}. In this case, since $d=1$, this is a scalar process. It is easy to compute $Z$ directly but we will use formula \eqref{Z solution} instead, as a demonstration. First, we note that $A$ defined in \eqref{def: A} will be $A(y) = \partial_y (-\lambda y + \sigma c) = -\lambda$. Thus, \eqref{Z solution} becomes
\[ Z_t = \int_0^t \exp(-\lambda(t-s))\frac{\sigma}{\delta} ds = \frac{\sigma}{\lambda\delta} (1-e^{-\lambda t}).\] 
We now have all the elements we need to write down the likelihood: from \eqref{likelihood 1}, we get
\begin{equation*}
L_{Y^{\mathcal D}}\left( y_{{\mathcal D}} |\ \theta\right) = L_{\Delta X_{\mathcal D}}\left( I_{\theta,{\mathcal D}}^{-1}(y_{{\mathcal D}})  \right)  \left( \frac{\lambda \delta}{\sigma(1-e^{-\lambda\delta})} \right)^N.
\end{equation*}
Finally, we note that the likelihood of the increments $\Delta X_{\mathcal D}$ is a mean zero Gaussian distribution with covariance matrix given by
\[ \left( \Sigma_h^{\mathcal D} \right)_{ij} = \frac{\delta^{2h}}{2}\left( |j-i+1|^{2h} + |j-i-1|^{2h} - 2|j-i|^{2h}\right),\ i,j = 1,\dots,N,\]
where $h$ is the Hurst parameter of the fractional Brownian motion. Thus, the likelihood becomes
\begin{equation}
\label{fOU:Lfinal}
 |2 \pi\Sigma_h^{\mathcal D}|^{-\frac{1}{2}}\exp\left( -\frac{1}{2}  I_{\theta,{\mathcal D}}^{-1}(y_{{\mathcal D}}) \left( \Sigma_h^{\mathcal D}\right)^{-1}  I_{\theta,{\mathcal D}}^{-1}(y_{{\mathcal D}})^* \right) \left( \frac{\lambda \delta}{\sigma(1-e^{-\lambda\delta})} \right)^N,
\end{equation}
where we denote by $z^*$ the transpose of a vector $z$. The corresponding log-likelihood is proportional to
\begin{equation}
\label{fOU logL1}
{\ell}_{Y(n)}\left( y_{{\mathcal D}} |\ \theta \right) \propto -\frac{1}{2}  I_{\theta,{\mathcal D}}^{-1}(y_{{\mathcal D}}) \left( \Sigma_h^{\mathcal D}\right)^{-1}  I_{\theta,{\mathcal D}}^{-1}(y_{{\mathcal D}})^* +N\log\left( \frac{\lambda\delta}{\sigma(1-e^{-\lambda\delta})} \right).
\end{equation}
Finally, we can replace $I_{\theta,{\mathcal D}}^{-1}$ above with its exact expression, which gives
\begin{equation}
\label{fOU logL}
{\ell}_{Y(n)}\left( y_{{\mathcal D}} |\ \lambda,\sigma \right) \propto -\frac{\lambda^2 \delta^2}{2\sigma^2(1-e^{-\lambda\delta})^2} \left( \Delta^\lambda y\right)_{{\mathcal D}} \left( \Sigma_h^{\mathcal D}\right)^{-1}  \left( \Delta^\lambda y\right)_{{\mathcal D}}^*+N\log\left( \frac{\lambda\delta}{\sigma(1-e^{-\lambda\delta})} \right),
\end{equation}
where by $\Delta^\lambda y_{k\delta} = y_{(k+1)\delta} - y_{k\delta}e^{-\lambda \delta}$. 

Now, let us construct the corresponding limiting likelihoods derived from \eqref{fOU logL}, for $h= \frac{1}{2}$, i.e. the diffusion case. Then, \eqref{fOU logL} becomes
\begin{eqnarray} 
\label{OU: logL}
\nonumber {\ell}_{Y(n)}\left( y_{{\mathcal D}(n)} |\ \lambda,\sigma \right) &=& -\frac{T}{2\delta}\log(2\pi)-\frac{\lambda^2 \delta}{2\sigma^2(1-e^{-\lambda\delta})^2} \left( \Delta^\lambda y\right)_{{\mathcal D}}  \left( \Delta^\lambda y\right)_{{\mathcal D}}^*\\
\nonumber &&+N\log\left( \frac{\lambda\delta}{\sigma(1-e^{-\lambda\delta})} \right) = \\
\nonumber &=& -\frac{T}{2\delta}\log(2\pi)-\frac{\lambda^2 \delta}{2\sigma^2(1-e^{-\lambda\delta})^2} \sum_{k=0}^{N-1}(y_{(k+1)\delta} - y_{k\delta}e^{-\lambda \delta})^2\\
\nonumber &&+N\log\left( \frac{\lambda\delta}{\sigma(1-e^{-\lambda\delta})} \right) = \\
 &=& -\frac{T}{2\delta}\log(2\pi)-\frac{T}{\delta}\log(\sigma) -\frac{T}{\delta}\log\left( \frac{1-e^{-\lambda\delta}}{\lambda\delta} \right)\\
\nonumber &&-\frac{\lambda^2 \delta}{2\sigma^2(1-e^{-\lambda\delta})^2} \sum_{k=0}^{N-1}(y_{(k+1)\delta} - y_{k\delta}e^{-\lambda \delta})^2.
\end{eqnarray}
Clearly, the first two terms of \eqref{OU: logL} is of order $o(\frac{1}{\delta})$. Noting that 
\[ \log\left( \frac{1-e^{-\lambda\delta}}{\lambda\delta} \right) = -\frac{\lambda}{2}\delta +{\mathcal O}(\delta^2), 
\]
it follows that the third term of \eqref{OU: logL} is of order ${\mathcal O}(1)$. In particular,
\[ \frac{T}{\delta}\log\left( \frac{1-e^{-\lambda\delta}}{\lambda\delta} \right) = -\frac{\lambda}{2}T +{\mathcal O}(\delta).
\]
Now, let us consider the final term in \eqref{OU: logL}. First, we note that
\[ \frac{\lambda^2 \delta}{2\sigma^2(1-e^{-\lambda\delta})^2} = \frac{1}{2\sigma^2 \delta} + \frac{\lambda}{2 \sigma^2} + {\mathcal O}(\delta).
\]
Then, we expand the sum as follows:
\begin{eqnarray*}
\sum_{k=0}^{N-1}(y_{(k+1)\delta} - y_{k\delta}e^{-\lambda \delta})^2 &=& \sum_{k=0}^{N-1}(y_{(k+1)\delta} - y_{k\delta})^2 + \sum_{k=0}^{N-1}y_{k\delta}^2(1-e^{-\lambda \delta})^2 \\
&&+ 2\sum_{k=0}^{N-1}y_{k\delta}(1-e^{-\lambda \delta})(y_{(k+1)\delta} - y_{k\delta}).
\end{eqnarray*}
The first term above is of order ${\mathcal O}(1)$, since it converges to the Quadratic Variation of the process, which is finite. The second term is of order ${\mathcal O}(\delta)$, with the dominating term being
\[ \sum_{k=0}^{N-1}y_{k\delta}^2(1-e^{-\lambda \delta})^2 = \lambda^2 \delta \sum_{k=0}^{N-1}y_{k\delta}^2\delta + {\mathcal O}(\delta^2),\]
since the latter sum converges to the corresponding integral. Finally, the last term of the sum expansion is also of order ${\mathcal O}(\delta)$, with the dominating term being
\[ 2\sum_{k=0}^{N-1}y_{k\delta}(1-e^{-\lambda \delta})(y_{(k+1)\delta} - y_{k\delta}) = 2 \lambda \delta \sum_{k=0}^{N-1}y_{k\delta}(y_{(k+1)\delta} - y_{k\delta})+{\mathcal O}(\delta^2),
\]
with the sum again converging to the corresponding It\^o integral. 
Putting everything together, we get that the final term of \eqref{OU: logL} will be
\[ \frac{1}{2 \sigma^2 \delta} \sum_{k=0}^{N-1}(y_{(k+1)\delta} - y_{k\delta})^2 + \frac{\lambda}{2 \sigma^2} \sum_{k=0}^{N-1}(y_{(k+1)\delta} - y_{k\delta})^2 + \frac{\lambda^2}{2\sigma^2} \sum_{k=0}^{N-1}y_{k\delta}^2\delta + \frac{\lambda}{\sigma^2}\sum_{k=0}^{N-1}y_{k\delta}(y_{(k+1)\delta} - y_{k\delta}) + {\mathcal O}(\delta).
\]
Finally, we get that the log-likelihood \eqref{OU: logL} can be expanded as
\begin{eqnarray}
\label{OU: logL exp}
\nonumber {\ell}_{Y}\left( y_{{\mathcal D}(n)} |\ \lambda,\sigma \right) &=& \frac{N}{T}\left( -\frac{T}{2}\log{2\pi\sigma^2}-\frac{1}{2 \sigma^2} \sum_{k=0}^{N-1}(y_{(k+1)\delta} - y_{k\delta})^2\right) + \\
\nonumber && \left(\frac{\lambda}{2}T-\frac{\lambda}{2 \sigma^2} \sum_{k=0}^{N-1}(y_{(k+1)\delta} - y_{k\delta})^2 \right)+ \\
&&\left(-\frac{\lambda^2}{2\sigma^2} \sum_{k=0}^{N-1}y_{k\delta}^2\delta - \frac{\lambda}{\sigma^2}\sum_{k=0}^{N-1}y_{k\delta}(y_{(k+1)\delta} - y_{k\delta}) 
\right) + {\mathcal O}(\frac{1}{N}),
\end{eqnarray}
where $N = \frac{T}{\delta}$. Thus, the normalised likelihoods are
\[ \ell^{(0)}\left( y_{{\mathcal D}(n)} |\ \lambda,\sigma \right) =   -\frac{1}{2}\log{2\pi\sigma^2}-\frac{1}{2 T \sigma^2} \sum_{k=0}^{N-1}(y_{(k+1)\delta} - y_{k\delta})^2
\] 
and
\[ \ell^{(1)}\left( y_{{\mathcal D}(n)} |\ \lambda,\sigma \right) =  \left(\frac{\lambda}{2}T-\frac{\lambda}{2 \sigma^2} \sum_{k=0}^{N-1}(y_{(k+1)\delta} - y_{k\delta})^2 \right)-\left(\frac{\lambda^2}{2\sigma^2} \sum_{k=0}^{N-1}y_{k\delta}^2\delta + \frac{\lambda}{\sigma^2}\sum_{k=0}^{N-1}y_{k\delta}(y_{(k+1)\delta} - y_{k\delta}) \right). 
\]
If we use these normalised likelihoods in the context of MLEs, clearly $\ell^{(0)}_{Y(n)}\left(y_{\mathcal{D}(n)}|\sigma,\lambda\right)$ depends only on $\sigma$, so parameter $\sigma$ is of order $0$ and maximisation leads to the estimate
\[ \hat{\sigma^2}(y_{\mathcal{D}(n)}) = \frac{1}{T}\sum_{k=0}^{N-1}(y_{(k+1)\delta} - y_{k\delta})^2.
\]
Parameter $\lambda$ will be of order 1 and can be estimated using $\ell^{(1)}_{Y(n)}\left(y_{\mathcal{D}(n)}|\hat{\sigma^2}(y_{\mathcal{D}(n)}),\lambda\right)$. This leads to the estimate
\[\hat{\lambda}(y_{\mathcal{D}(n)}) = -\frac{\sum_{k=0}^{N-1}y_{k\delta}(y_{(k+1)\delta} - y_{k\delta})}{\sum_{k=0}^{N-1}y_{k\delta}^2\delta}.
\]
It is reassuring to see both estimates are known to be consistent and similar (up to discretisation) with MLE estimates one gets using standard likelihood construction. Also, it is worth noting that $\ell^{(1)}_{Y(n)}\left(y_{\mathcal{D}(n)}|\hat{\sigma^2}(y_{\mathcal{D}(n)}),\lambda\right)$ in the limit coincides with the likelihood constructed using Girsanov, which of course also requires knowledge of the diffusion parameter $\sigma$ \cite{iacus2009simulation}.

\section{Convergence of Approximate Likelihood}
\label{section: convergence}

In this section, we study the behaviour of the approximate likelihoods constructed in section \ref{section: limit}. We will make the following assumptions
\begin{assumption}
\label{ass: distribution}
Let $X$ be the stochastic process driving \eqref{eq:main}, defined on the probability space $(\Omega,\mathcal{F},\mathbb{P})$.
\begin{itemize}
\item[(i)] We assume that the distribution of the increments $\Delta X_{\mathcal{D}(n)}$ is absolutely continuous with respect to Lebesgue $\forall n>0$ and we define the log-likelihood $\ell_{\Delta X_{\mathcal{D}(n)}}$ as the logarithm of the corresponding Radon-Nikodym derivative. 
\item[(ii)] We assume that 
\[ d_p(\pi_n(X),X)\to 0,\ {\rm as}\ n\to\infty, \mathbb{P}-a.s.,\]
where the lift of $X$ to a rough path is defined as the limit to the lifts of $\pi_n(X)$, which converge. 
\item[(iii)] Moreover, we assume that the log-likelihood satisfies 
\begin{equation}
\label{cont of log likelihood}
\left| \ell_{\Delta X_{{\mathcal D}(n)}} \left(\Delta x_{{\mathcal D}(n)}\right) - \ell_{\Delta X_{{\mathcal D}(n)}}  \left(\Delta \tilde x_{{\mathcal D}(n)}\right) \right|  \leq \phi(N(n)) \psi\left(d_p(x,\tilde{x})\right),
\end{equation}
where $\Delta x_{{\mathcal D}(n)}$ and $\Delta \tilde x_{{\mathcal D}(n)}$ denote increments on ${\mathcal D}(n)$ respectively of $p$-rough paths $x$ and $\tilde{x}$. Functions $\phi$ and $\psi$ are real-valued functions, with $\psi$ increasing and satisfying
\[ \psi(N(n)) d_p(\pi_n(X),X)^2 \to 0,\]
in the same sense as above.
\end{itemize}
\end{assumption}
\begin{assumption}
\label{ass: model}
\begin{itemize}
\item[(i)] We assume that $a(\cdot;\theta)$ and $b(\cdot;\theta)$ are both $\rm{Lip}(\gamma+1)$ uniformly in $\theta$, for some $\gamma>p$. This implies uniform in $\theta$ continuity of integration and the It\^o map.
\item[(ii)] We assume that $b$ is bounded away from $0$, i.e.
\begin{equation}
\label{b lower bound} 
\inf_{y,\theta}\left| |b(y;\theta)| \right|= \frac{1}{M_b}>0.
\end{equation}
\end{itemize}
\end{assumption}
The first assumption concerns the distribution of the driving process $X$ while the second one concerns the model. Under these assumptions, we will prove the following:
\begin{theorem}
\label{th: main}
Let $y$ be the response to a realisation $x$ of a $p$-rough path $X$ through \eqref{eq:main2} and $y(n)$ be the response to $\pi_n(x)$ through \eqref{eq:main}, where $\pi_n(x)$ is the piecewise linear interpolation of $x$ on grid ${\mathcal D}(n) = \{ k 2^{-n}T,\ k=0,\dots,N \}$ for $N= 2^n T$. Suppose that assumptions \ref{ass: distribution} and \ref{ass: model} are satisfied and let us define $\ell_{Y(n)} \left(\cdot|\theta\right)$ as the logarithm of \eqref{likelihood 1}.  Then,
\begin{equation}
\lim_{n\to\infty} \sup_{\theta}\left| \ell_{Y(n)} \left(y_{{\mathcal D}(n)}|\theta\right) - \ell_{Y(n)} \left(y(n)_{{\mathcal D}(n)}|\theta\right)\right| = 0. 
\end{equation}
\end{theorem}

This theorem allows us to transfer any consistency properties of the estimators corresponding to data from model \eqref{eq:main}, where the likelihood is exact. Proving consistency requires specific knowledge of the distribution. In this paper, we build a general framework and we only concentrate on the error due to the approximation error in the construction of the likelihood. If we use the scaled likelihoods to construct MLEs, the following corollary is crucial for proving consistency:

\begin{corollary}
Suppose the the assumptions of theorem \eqref{th: main} are satisfied. Let us also assume that assumption \ref{convergence of scales} is satisfied. Then, the following holds for all scaled likelihoods $ \ell_{Y(n)}^{(k)} \left(y_{{\mathcal D}(n)}|\theta\right)$ with $\alpha_k\leq 0$:
\[ \lim_{n\to\infty} \sup_{\theta}\left| \ell_{Y(n)}^{(k)} \left(y_{{\mathcal D}(n)}|\theta\right) - \ell_{Y(n)}^{(k)} \left(y(n)_{{\mathcal D}(n)}|\theta\right)\right| = 0.  \]
\end{corollary}

Before proving the theorem, let us set some notation. We define
\begin{equation}
\label{inverse Ito}
I_{\theta,n}^{-1}(y_{{\mathcal D}(n)}) = I_\theta^{-1}(Y(n,y_{{\mathcal D}(n)})),
\end{equation}
where, as before, $Y(n,y_{{\mathcal D}(n)})$ is the response to a piecewise linear path parametrised by its values on the grid, $y_{{\mathcal D}(n)}$. So, $I_\theta^{-1}(Y(n,y_{{\mathcal D}(n)}))$ will be exactly that piecewise linear path whose response, when driving the system matches the observations $y_{{\mathcal D}(n)}$. $I_{\theta,n}^{-1}(y(n)_{{\mathcal D}(n)})$ and $Y(n,y(n)_{{\mathcal D}(n)})$ corresponding to observations $y(n)_{{\mathcal D}(n)}$ are defined similarly.

To prove the theorem, we will need the following two lemmas:

\begin{lemma}
\label{log error control}
Let $Z_{t_{i+1}-t_i}$ and $I_{\theta,n}^{-1}$ be defined as in \eqref{Z solution} and \eqref{inverse Ito} respectively and suppose that assumption \ref{ass: model} is satisfied. Then
\begin{eqnarray*}
&\left|\sum_{t_i\in{\mathcal D}(n)}\log |Z_{t_{i+1}-t_i}\left(  I_{\theta,n}^{-1} (y_{{\mathcal D}(n)})_{t_i}\right)| - \sum_{t_i\in{\mathcal D}(n)}\log |Z_{t_{i+1}-t_i}\left(  I_{\theta,n}^{-1} ( y(n)_{{\mathcal D}(n)})_{t_i}\right)|\right| \leq \\
&C \cdot d_p ( I_{\theta,n}^{-1} (y_{{\mathcal D}(n)}),  I_{\theta,n}^{-1} (y(n)_{{\mathcal D}(n)}))
\end{eqnarray*}
for some $C \in \R_+$ depending on Lipschitz bounds on $a, b$, $M_b$ and $T$.
\end{lemma}

\begin{proof}
We write
\begin{eqnarray}
\label{log error}
&\sum_{t_i\in{\mathcal D}(n)}\log |Z_{t_{i+1}-t_i}\left(  I_{\theta,n}^{-1} (y_{{\mathcal D}(n)})_{t_i}\right)| - \sum_{t_i\in{\mathcal D}(n)}\log |Z_{t_{i+1}-t_i}\left(  I_{\theta,n}^{-1} ( y(n)_{{\mathcal D}(n)})_{t_i}\right)| = \\
&\nonumber = \sum_{t_i\in{\mathcal D}(n)}\log \frac{|Z_{t_{i+1}-t_i}\left(  I_{\theta,n}^{-1} (y_{{\mathcal D}(n)})_{t_i}\right)| }{ |Z_{t_{i+1}-t_i}\left(  I_{\theta,n}^{-1} (y(n)_{{\mathcal D}(n)})_{t_i}\right)| }
\end{eqnarray}
As before, using the continuity of integrated function within $Z_{t_{i+1}-t_i}$ with respect to the time variable, we write
\begin{equation*}
Z_{t_{i+1}-t_i}\left(  I_{\theta,n}^{-1} (y_{{\mathcal D}(n)})_{t_i}\right) = \exp\left( {\bf A}(y_{t_i},I_{\theta,n}^{-1} (y_{{\mathcal D}(n)})_{t_i};\theta)\right)_{\zeta_i,t_{i+1}} \cdot b(F_{\zeta_i};\theta)  \cdot (t_{i+1}-t_i)
\end{equation*}
and
\begin{equation*}
Z_{t_{i+1}-t_i}\left(  I_{\theta,n}^{-1} (y(n)_{{\mathcal D}(n)})_{t_i}\right) = \exp\left( {\bf A}(y(n)_{t_i},I_{\theta,n}^{-1} (y(n)_{{\mathcal D}(n)})_{t_i};\theta)\right)_{\eta_i,t_{i+1}} \cdot b(F_{\eta_i};\theta)  \cdot (t_{i+1}-t_i)
\end{equation*}
for some $\zeta_i,\eta_i\in[t_i,t_{i+1}]$.  So, \eqref{log error} simplifies to
\begin{eqnarray*}
&\sum_{t_i\in{\mathcal D}(n)}\log \frac{|\exp\left( {\bf A}(y_{t_i},I_{\theta,n}^{-1} (y_{{\mathcal D}(n)})_{t_i};\theta)\right)_{\zeta_i,t_{i+1}} \cdot b(F(y_{t_i},I_{\theta,n}^{-1} (y_{{\mathcal D}(n)})_{t_i})_{\zeta_i};\theta)  | }{ |\exp\left( {\bf A}(y(n)_{t_i},I_{\theta,n}^{-1} (y(n)_{{\mathcal D}(n)})_{t_i};\theta)\right)_{\eta_i,t_{i+1}} \cdot b(F(y(n)_{t_i},I_{\theta,n}^{-1} (y(n)_{{\mathcal D}(n)})_{t_i})_{\eta_i};\theta) | } = \\
&\sum_{t_i\in{\mathcal D}(n)}\log \frac{|\exp\left( {\bf A}(y_{t_i},I_{\theta,n}^{-1} (y_{{\mathcal D}(n)})_{t_i};\theta)\right)_{\zeta_i,t_{i+1}}  | }{ |\exp\left( {\bf A}(y(n)_{t_i},I_{\theta,n}^{-1} (y(n)_{{\mathcal D}(n)})_{t_i};\theta)\right)_{\eta_i,t_{i+1}}| }  + \sum_{t_i\in{\mathcal D}(n)}\log \frac{|b(F(y_{t_i},I_{\theta,n}^{-1} (y_{{\mathcal D}(n)})_{t_i})_{\zeta_i};\theta)  | }{ | b(F( y(n)_{t_i},I_{\theta,n}^{-1} (y(n)_{{\mathcal D}(n)})_{t_i})_{\eta_i};\theta) | }.
\end{eqnarray*}
Focusing on the second summand, we write
\begin{eqnarray*}
&|\sum_{t_i\in{\mathcal D}(n)}\log \frac{|b(F(y_{t_i},I_{\theta,n}^{-1} (y_{{\mathcal D}(n)})_{t_i})_{\zeta_i};\theta)  | }{ | b(F(y(n)_{t_i},I_{\theta,n}^{-1} (y(n)_{{\mathcal D}(n)})_{t_i})_{\eta_i};\theta) | }| \leq \\ 
&\sum_{t_i\in{\mathcal D}(n)} \log\left( 1 + M_b\left| |b(F(y_{t_i},I_{\theta,n}^{-1} (y_{{\mathcal D}(n)})_{t_i})_{\zeta_i};\theta)|-|b(F(y(n)_{t_i},I_{\theta,n}^{-1} (y(n)_{{\mathcal D}(n)})_{t_i})_{\eta_i};\theta)|\right|\right) \leq \\ 
&M_b \cdot \sum_{t_i\in{\mathcal D}(n)} \left| |b(F(y_{t_i},I_{\theta,n}^{-1} (y_{{\mathcal D}(n)})_{t_i})_{\zeta_i};\theta)|-|b(F(y(n)_{t_i},I_{\theta,n}^{-1} (y(n)_{{\mathcal D}(n)})_{t_i})_{\eta_i};\theta)|\right| \leq \\
&M_bC_b  d_p ( I_{\theta,n}^{-1} (y_{{\mathcal D}(n)}),  I_{\theta,n}^{-1} (y(n)_{{\mathcal D}(n)})),
\end{eqnarray*}
where we used the inequality $\log(1+x)<x$, assumption that $\inf_y\left| |b(y)| \right|= \frac{1}{M_b}>0$, the Lipschitz continuity of $b$ and the universal limit theorem (see \cite{Baudoinrough} for exact bound). Similarly, we get that 
\begin{equation*}
\sum_{t_i\in{\mathcal D}(n)}\log \frac{|\exp\left( {\bf A}(y_{t_i},I_{\theta,n}^{-1} (y_{{\mathcal D}(n)})_{t_i};\theta)\right)_{\zeta_i,t_{i+1}}  | }{ |\exp\left( {\bf A}(y(n)_{t_i},I_{\theta,n}^{-1} ( y(n)_{{\mathcal D}(n)})_{t_i};\theta)\right)_{\eta_i,t_{i+1}}| }  \leq M_A C_A d_p ( I_{\theta,n}^{-1} (y_{{\mathcal D}(n)}),  I_{\theta,n}^{-1} (y(n)_{{\mathcal D}(n)})).
\end{equation*}
noting that $inf_y |exp(A)(y;\theta)_{t_i,t_{i+1}}| >0$ is always true. Putting the two together, we get the result.

\end{proof}

\begin{lemma}
\label{lemma: uniform limit}
Let $I_\theta^{-1}$ be the inverse It\^o map defined by system \eqref{eq:main}, satisfying assumptions \ref{ass: model}. We denote by $ Y(n,I_{\theta_0}(x)_{{\mathcal D}(n)})$ and $Y(n,I_{\theta_0}(\pi_n(x))_{{\mathcal D}(n)})$ the responses to the piecewise linear map as in \eqref{eq:main2}, parametrised by its values on the grid ${\mathcal D}(n)$, given by $I_{\theta_0}(x)_{{\mathcal D}(n)}$ and $I_{\theta_0}(\pi_n(x))_{{\mathcal D}(n)}$ respectively, where $x$ is a fixed rough path in $G\Omega_p(\R^d)$ and $\theta_0\in\Theta$. Then, 
\begin{equation}
\label{inverse distance bound}
\lim_{n\to\infty}  \sup_\theta d_p\left( I_{\theta}^{-1} \left(Y(n,I_{\theta_0}(x)_{{\mathcal D}(n)})\right), I_{\theta}^{-1} \left(Y(n,I_{\theta_0}(\pi_n(x))_{{\mathcal D}(n)})\right)\right) = 0,
\end{equation}
provided that $d_p(\pi_n(x),x)\to 0$ as $n\to\infty$.
\end{lemma}
\begin{proof}
Under the assumption that $b$ is invertible, $I_\theta^{-1}$ can be expressed as an integral of the rough path $Y(n,\cdot)$. Thus, it is sufficient to show that
\begin{equation}
\label{distance 2}
\lim_{n\to\infty}  d_p\left( I_{\theta_0}^{-1} \left(Y(n,I_{\theta_0}(x)_{{\mathcal D}(n)})\right), I_{\theta_0}^{-1} \left(Y(n,I_{\theta_0}(\pi_n(x))_{{\mathcal D}(n)})\right)\right) = 0,
\end{equation}
as a consequence of the uniform in $\theta$ continuity of integration and the It\^o map in $p$-variation topology. By construction, these are piecewise linear paths. Clearly, $\pi_n(x) = I_{\theta_0}^{-1} \left(Y(n,I_{\theta_0}(\pi_n(x))_{{\mathcal D}(n)})\right)$ and $\pi_n(x)_{t_i,t_{i+1}} = \Delta x_{i}$. Let us denote by
 \[ \tilde{x}(n) = I_{\theta_0}^{-1} \left(Y(n,I_{\theta_0}(x)_{{\mathcal D}(n)})\right)
 \]
 and let $\tilde{x}(n)_{t_i,t_{i+1}} = \Delta \tilde{x}_{i}$ be the corresponding increments. Then, \eqref{distance 2} becomes equivalent to
 \begin{equation}
\lim_{n\to\infty}  \| \tilde{x}(n) - \pi_n(x) \|_{p-var} = 0,
\end{equation}
where $\|\cdot\|_{p-var}$ is the p-variation metric. The process $\tilde{x}(n) - \pi_n(x)$ is piecewise linear on ${\mathcal D}(n)$, as a difference of two piecewise linear functions on ${\mathcal D}(n)$. We know that in this case, the supremum of the $p$-variation norm is achieved at a subset of ${\mathcal D}(n)$ (see \cite{Driver13}), so
\begin{equation}
\label{piecewise difference} 
\| \tilde{x}(n) - \pi_n(x) \|_{p-var} = \left( \sup_{{\mathcal E}(n) \subset {\mathcal D}(n)} \sum_{\tau_i \in{\mathcal E}(n) \subset {\mathcal D}(n)} \| \tilde{x}(n)_{\tau_i,\tau_{i+1}} - {\pi_n(x)}_{\tau_i,\tau_{i+1}} \|^p \right)^\frac{1}{p}.
\end{equation}
Now, by definition, we have that 
\[ y_{t_{i+1}} = F_{\theta_0}(y_{t_i}, \delta; \Delta \tilde{x}_i),\]
where $F$ denotes the solution of the ODE \eqref{general ODE}, as before. By expanding $F_{\theta_0}$ around $\Delta x_i$, we get
\begin{equation*}
y_{t_{i+1}} = F_{\theta_0}(y_{t_i}, \delta; \Delta {x}_i) + D_c F_{\theta_0}(y_{t_i}, \delta; \xi_i) (\Delta \tilde{x}_i - \Delta x_i),
\end{equation*}
where $D_c F_{\theta_0}$ is the derivative of the solution to the ODE with respect to constant $c$. It follows that
\[ \Delta \tilde{x}_i - \Delta x_i = D_c F_{\theta_0}(y_{t_i}, \delta; \xi_i)^{-1}\left( y_{t_{i+1}} - F_{\theta_0}(y_{t_i}, \delta; \Delta {x}_i) \right).\]
We have already seen that $D_c F_{\theta_0}$ is given by the process $Z(c)$ defined in \eqref{Z definition}. So, the above relationship becomes
\[ \Delta \tilde{x}_i - \Delta x_i = Z(\xi_i)^{-1}\left( y_{t_{i+1}} - F_{\theta_0}(y_{t_i}, \delta; \Delta {x}_i) \right).\]
Moreover, it follows by \eqref{Z solution} that if $b$ is uniformly bounded away from zero, i.e. $\|b\|>\frac{1}{M}$, then so is $Z$ and consequently, $Z^{-1}$ is bounded from above. Thus, \eqref{piecewise difference} becomes
\begin{eqnarray*}
\| \tilde{x}(n) - \pi_n(x) \|_{p-var} \leq M \left( \sum_{\tau_i \in{\mathcal E}(n)} \| \sum_{\tau_i\leq t_i < \tau_{i+1}}\left(y_{t_{i+1}} - F_{\theta_0}(y_{t_i}, \delta; \Delta {x}_i)\right) \|^p \right)^\frac{1}{p}.
\end{eqnarray*}
Finally, we note that $y_{t_{i+1}}$ and  $F_{\theta_0}(y_{t_i}, \delta; \Delta {x}_i)$ are both the solution of \eqref{eq:main2} for the same initial conditions $y_{t_i}$, driven by $x$ and $\pi_n(x)$ respectively. Since the first iterated integral (increment) of these two drives is the same, if we expand $y_{t_{i+1}}$ and  $F_{\theta_0}(y_{t_i}, \delta; \Delta {x}_i)$ in terms of the iterated integrals of the driver and initial conditions, we find that their difference comes from the difference of the second and higher iterated integrals and thus it can be bounded by $d_p(x,\pi_n(x))^2$. 
\end{proof}

The two lemmas together prove the theorem.

\end{document}